\documentclass[12pt]{amsart}

\usepackage[dvips]{graphics}
\usepackage{color,amssymb,amsbsy,amsmath,amsfonts,amssymb,
amscd,epsfig,times,graphics}

\def\Re{{\sf Re}\,}

\newtheorem{theorem}{Theorem}[section]
\newtheorem{proposition}[theorem]{Proposition}

\newtheorem{question}[theorem]{Question}
\theoremstyle{definition}
\newtheorem{definition}[theorem]{Definition}

\numberwithin{equation}{section}

\newcommand{\C}{\mathbb C}
\newcommand{\Cn}{\mathbb C^n}

\newcommand{\N}{\mathbb N}
\newcommand{\B}{\mathbb B}
\newcommand{\Bn}{\mathbb B^n}
\newcommand{\D}{\mathbb D}

\begin{document}
\begin{abstract}
In this paper we study the class of ``shearing'' holomorphic maps of the unit ball of the form $(z,w)\mapsto (z+g(w), w)$. Besides general properties, we use such maps to construct an example of a normalized univalent map of the ball onto a Runge domain in $\C^n$ which however cannot be embedded into a Loewner chain whose range is $\C^n$.
\end{abstract}

\title[Shearing and embedding]{Shearing maps and a Runge map of the unit ball which does not embed  into a Loewner chain with range $\C^n$}

\dedicatory{Dedicated to the memory of  our friend Prof. Gabriela Kohr}

\author[F. Bracci]{Filippo Bracci$^\dag$}
\address{F. Bracci: Dipartimento di Matematica, Universit\`a di Roma ``Tor Vergata", Via della Ricerca
Scientifica 1, 00133, Roma, Italia.} \email{fbracci@mat.uniroma2.it}

\author[P. Gumenyuk]{Pavel Gumenyuk}
\address{Department of Mathematics, Politecnico di Milano, via E. Bonardi 9, 20133, Milan, Italy} \email{pavel.gumenyuk@polimi.it}


\thanks{$^1$Partially supported by PRIN 2017 Real and Complex Manifolds: Topology, Geometry and holomorphic
dynamics, Ref: 2017JZ2SW5, by GNSAGA of INdAM and by the MIUR Excellence Department Project awarded
to the Department of Mathematics, University of Rome Tor Vergata, CUP E83C18000100006.}
\maketitle

\tableofcontents

\newcommand{\di}[2]{d#1(#2)}

\section{Introduction}

Let $\Bn:=\{z\in \C^n: \|z\|<1\}$ be the Euclidean unit ball in $\C^n$ and $\D:=\B^1$.
Let
$$H(\Bn):=\big\{f:\Bn\to\Cn:\,f\mbox{ is  holomorphic}\big\}.$$ 
As usual, we endow $H(\Bn)$ with the topology of uniform convergence on compacta. 
We say that $f\in H(\B^n)$ is normalized if $f(0)=0$ and $\di{f}{0}={\sf Id}$.
Let
$$S(\Bn):=\big\{f\in H(\Bn):\,f\mbox{ is  normalized univalent on }
\B^n\}.
$$
Also, let
\[
U_0(\C^n):=\{f:\C^n\to \C^n \hbox{\ univalent}, f(0)=0, \di{f}{0}={\sf Id}\},
\]
and
\[
A_0(\C^n):=\{f\in U_0(\C^n): f \hbox{\ is surjective}\}.
\]
Note that $A_0(\C^n)$ consists of all the automorphisms of $\C^n$ tangent to the identity at the origin, and that for $n>1$ such a group is huge: in fact, for every jet of the form
\[
{\sf Id}+\sum_{j=2}^kP_j,
\]
where $k\in \N$, $k\geq 2$ and $P_j$ is a homogeneous polynomial vector of degree $j$, there exists  $F\in A_0(\C^n)$ such that $F={\sf Id}+\sum_{j=2}^kP_j+o(k)$ (see \cite{Fo}). Also, clearly,  $A_0(\C^n)\circ S(\B^n)=S(\B^n)$.

Clearly, $A_0(\C)=U_0(\C)=\{\sf Id\}$, while, for $n>1$,  $A_0(\C^n)\subsetneq U_0(\C^n)$, due to the existence of the so-called Fatou-Bieberbach domains.

As it is well known, $S(\D)$ is compact in $H(\D)$. The compactness of a family of holomorphic maps is a strong property, because, for instance, it allows growth estimates, estimates on the differentials, existence of support points, Koebe's type theorem, Bieberbach's conjecture, and so on.

In higher dimension the family $S(\Bn)$ is not compact. (For instance, for every $A>0$, the restriction to $\Bn$ of the automorphism $(z,w)\mapsto (z+Aw^2, w)$ belongs to $S(\Bn)$ but, for $A\to+\infty$ there is no limit point.) It is then natural to ask:
\begin{question}
Is there a ``natural'' compact set $\mathcal K\subseteq S(\B^n)$ such that
\[
A_0(\C^n)\circ \mathcal K=S(\B^n)?
\]
\end{question}

For $n=1$, clearly $\mathcal K=S(\D)$. For $n>1$, several natural compact subclasses of $S(\B^n)$ has been introduced. We need some preliminaries in order to define one of the most natural.

 \begin{definition}
\label{Lch}
A family $(f_t)_{t\ge 0}$ of holomorphic
mappings on $\B^n$ is called a {\sl Loewner chain} if $\{e^{-t}f_t\}_{t\ge
0}\subseteq S(\Bn)$ and $f_s(\Bn)\subseteq f_t(\Bn)$, $0\le s\le
t<\infty$.

If, in addition, $\{e^{-t}f_t\}_{t\ge 0}$ is a normal family, then
$(f_t)_{t\ge0}$ is called a {\sl normal Loewner chain}.
\end{definition}

For a Loewner chain $(f_t)_{t\ge 0}$ the set
\[
R(f_t):=\bigcup_{t\ge0}f_t(\Bn)
\]
is called the {\sl Loewner range} of $(f_t)_{t\ge 0}$.\medskip

We say that a mapping $f\in S(\Bn)$ \textsl{embeds into a Loewner chain}
$(f_t)_{t\ge0}$ if $f_0=f$. Let
\begin{flalign*}
& \mbox{~}\hskip-7em &  S^0(\Bn)&:=\big\{f\in S(\Bn): f\mbox{ embeds into a normal Loewner
chain }(f_t)_{t\ge0}\big\},\\
& \mbox{~}\hskip-7em &  \hskip-4em S^1(\Bn)&:=\big\{f\in S(\Bn): f\mbox{ embeds into a Loewner
chain }(f_t)_{t\ge0}\big\},\\
& \text{\,\hskip\parindent and} \hskip-7em &
S_R(\Bn)&:=\big\{f\in S(\Bn): f(\Bn)\mbox{ is a Runge
domain}\big\}.
\end{flalign*}

For $n=1$, by the so-called Pommerenke's embedding theorem and Runge's theorem, $S(\D)=S^0(\D)=S^1(\D)=S_R(\D)$ (see \cite{Pom}). This fact suggested to Gabriela Kohr \cite{Gabi} and I.\,Graham, H.\,Hamada and G.\,Kohr \cite{GHK-CJM} (see also \cite{GK}) to define the remarkable class $S^0(\B^n)$. Among many other properties, they showed that $S^0(\B^n)$ is compact, and every normalized convex or starlike mapping of $\B^n$ is contained in $S^0(\B^n)$. Moreover, we have

\begin{theorem}[Graham, Kohr,  Pfaltzgraff \cite{GKP}] \label{t1}
Let $(f_t)_{t\ge0}$ be a Loewner chain. Then there exist $\phi\in U_0(\C^n)$ and a normal Loewner chain
$(g_t)_{t\ge0}$ such that $f_t=\phi\circ g_t$, for all $t\ge0.$ In
particular,
$$S^1(\Bn)=U_0(\C^n)\circ S^0(\B^n).$$
\end{theorem}
However, although the above theorem is very tempting to make guess that one can take $\mathcal K=S^0(\B^n)$, it turns out that
\[
A_0(\C^n)\circ S^0(\B^n)\subsetneq S^1(\B^n).
\]
Indeed,  let $(g_t)$ be a normal Loewner chain. By \cite[Prop. 2.1]{ABW}, the Loewner range of $(g_t)$ is $\C^n$ and thus, by a theorem of Docquier-Grauert \cite{DocG}, every $f\in S^0(\B^n)$ has Runge image in $\C^n$. Therefore, also every $h\in A_0(\C^n)\circ S^0(\B^n)$ has Runge image in $\C^n$. However, there exists $g\in S^1(\Bn)$ such that $g(\B^n)$ is not Runge (see \cite[Example 2.2]{ABW}).

Constructing an example of a biholomorphic image of the ball $\B^3$ whose image is not Runge in any bigger domain, J. E. Forn\ae ss and E. Wold \cite{FW} recently proved that
\[
S(\B^3)\neq S^1(\B^3),
\]
which shows, in particular, that, at least for $n\geq 3$, the compact class $\mathcal K$ has to be rather exotic.

As the above discussion shows, the main issues are due to Runge-ness property. Thus one might ask what happens if we restrict to the class  $S_R(\B^n)$.

Since $S^0(\B^n)\subset S_R(\B^n)$, it is thus natural to ask (cfr. Question Q1) in \cite{ABW}):

\begin{question}
Is it true that $S_R(\B^n)=A_0(\C^n)\circ S^0(\B^n)$?
\end{question}

The main result of this note is the following:

\begin{theorem}\label{main}
There exists a map $f\in S_R(\B^2)$ that cannot be embedded into a Loewner chain whose range is $\C^2$.  In particular, $A_0(\C^2)\circ S^0(\B^2)\subsetneq S_R(\B^2)$.
\end{theorem}

This result can be easily generalized to any $n\geq 3$; hence, for $n>1$, $A_0(\C^n)\circ S^0(\B^n)\subsetneq S_R(\B^n)$.
The proof relies on the construction of an example of the form $\B^2\ni (z,w)\mapsto (z+g(w), w)$, where $g:\D \to \C$ is holomorphic and $g(0)=g'(0)=0$. Thus, in Section~2 we study general properties of the class of such maps, which we call \textsl{shearing maps}. In Section~3, we prove first the following new growth estimate for the differential of maps in $S^0(\B^n)$. For a linear map $L:\C^n\to\C^n$, we denote by $\|L\|$ the operator norm of~$L$.
\begin{proposition}
\label{dif-ineq}
Let $f\in S^0(\B^n)$, $n\ge1$. Then for every $r\in(0,1)$ the following
inequality holds:
$$\big\|\di{f}{z}\big\|\le\frac{(1+\sqrt{r})^2}{(1-r)^3},\quad \|z\|\le
r.$$
\end{proposition}

Then, choosing  a (non-normal, in the sense of geometric function theory) function $g$, we obtain an example of a function in $S_R(\B^2)$ such that  the  growth of the differential is faster than predicted for maps in $A_0(\B^2)\circ S^0(\B^2)$.

\bigskip

The following question remains open:

\begin{question}
Is it true that $S_R(\B^n)\subset U_0(\C^n)\circ S^0(\B^n)$? Equivalently, due to Theorem~\ref{t1}, is it true that every $f\in S_R(\B^n)$ embeds into some Loewner chain, possibly with range different from~$\C^n$?
\end{question}

\medskip

The thoughts on which this work is based, germinated in Cluj in 2015 when the first author was visiting Gabriela Kohr. We separated with the promise to continue to work on this subject together. As often happens, years passed by,  the material left in a drawer, and, when the sad news of the premature departure of Gabi arrived, we could only realize that the time was over and we could never benefit again of the amazing ideas of Gabi. We can just thank Gabi for  sharing with us her enthusiasm, her deep intuitions, her constant support and many years of friendship.

The authors also warmly thank Mihai Iancu for very fruitful discussions on the subject.

\section{The class of shearing maps in $S_R(\B^2)$}

Fix a holomorphic function $g:\D\to\C$ with $g(0)=g'(0)=0$. Let  $f:\B^2\to\C^2$ be given by
$$f(z):=(z_1+g(z_2),z_2),\;z:=(z_1,z_2)\in\B^2.$$
Then $f\in S(\B^2)$. If $g$ is an entire function, then $f$ extends
to an automorphism of $\C^2$.

\begin{proposition}
\label{p4}
$f\in S_R(\B^2)$.
\end{proposition}

\begin{proof}
Consider the Taylor expansion for $g$:
$g(\zeta)=\sum_{k=2}^{\infty} a_k \zeta^k$,
$\zeta\in\D$.
For every $m\in\mathbb{N}$, $m\ge2$, let
$g_m(\zeta):=\sum_{k=2}^m a_k \zeta^k$, $\zeta\in\C$, and
$f_m(z):=(z_1+g_m(z_2),z_2),\;z=(z_1,z_2)\in\C^2$. Since, for every
$m\in\mathbb{N}$, $f_m$ is an automorphism of $\C^2$, we deduce that
$f_m\big|_{\B^2}\in S_R(\B^2)$ (see \cite{ABW}). Thus $f\in S_R(\B^2)$,
because $S_R(\B^2)$ is closed in~$H(\B^2)$ (see \cite{ABW}).
\end{proof}

\begin{proposition}
\label{p5}
Assume that $g(\zeta)=\sum_{k=2}^{\infty} a_k \zeta^k$,
$\zeta\in\D$, and $\sum_{k=2}^{\infty} k|a_k|<\infty$. Then
$f$ embeds into a Loewner chain with range $\C^2$. In particular, $f\in
 S^1(\B^2)$.
\end{proposition}

\begin{proof}
As before, let $g_m(\zeta):=\sum_{k=2}^m a_k \zeta^k$,
$\zeta\in\C$,  and
$f_m(z):=(z_1+g_m(z_2),z_2),\;z=(z_1,z_2)\in\C^2$. Then $f_m$ extends
to an automorphism of $\C^2$ and
$f_m^{-1}(z)=(z_1-g_m(z_2),z_2),\;z=(z_1,z_2)\in\C^2$, for all
$m\in\mathbb{N}$.
Hence $(f_m^{-1}\circ
f)(z)=(z_1+\sum_{k=m+1}^{\infty}a_kz_2^k,z_2),\;z=(z_1,z_2)\in\C^2
$, for all $m\in\mathbb{N}$. Let $N\in\mathbb{N}$ be sufficiently large
such that $\sum_{k=N+1}^{\infty} k|a_k|\le 1$. Then, in view of
\cite[Lemma 2.2]{GHK-MN}, we have that $f_N^{-1}\circ f\in S^0(\B^2)$.
Since $f_N$ is an automorphism of $\C^2$, we conclude that $f$ embeds
into a Loewner chain with range $\C^2$, by Theorem \ref{t1}.
\end{proof}

Denote by $S^*(\Bn)$ the subset of $S(\B^n)$ consisting of all starlike
mappings of $\Bn$.

\begin{proposition}
\label{p-new1}
Assume that $g(\zeta)=\sum_{k=2}^{\infty} a_k \zeta^k$,
$\zeta\in\D$, and $\sum_{k=2}^{\infty} (k-1)|a_k|\leq \frac{3\sqrt{3}}{2}$. Then $f\in
S^*(\B^2)$. In particular, $f\in S^0(\B^2)$.
The constant $\frac{3\sqrt{3}}{2}$ is sharp.
\end{proposition}

\begin{proof}
Since we have
\begin{eqnarray*}
\Re \big\langle [\di{f}{z}]^{-1}f(z), z\big\rangle
&=&\| z\|^2+\Re\sum_{k=2}^{\infty}(1-k)a_kz_2^k\overline{z}_1
\\
&\geq & \| z\|^2-\frac{2}{3\sqrt{3}}\| z\|^2\sum_{k=2}^{\infty}(k-1)|a_k|
\\
&\ge & 0
\end{eqnarray*}
for all $z\in \mathbb{B}^2\setminus \{ 0\}$,
$f$ is starlike.

By \cite{B15}, if $f\in S^0(\B^2)$, then $|a_2|\leq \frac{3\sqrt{3}}{2}$. In particular, the map $(z_1,z_2)\mapsto(z_1+a_2z_2^2, z_2)$ does not belong to $S^\ast(\B^2)$ if $|a_2|>\frac{3\sqrt{3}}{2}$. Hence, $\frac{3\sqrt{3}}{2}$ is sharp.
\end{proof}

\begin{proposition}
\label{p6}
If $f\in S^*(\B^2)$, then $g$ is bounded.
\end{proposition}

\begin{proof}
First, we prove the following: if $f\in S^*(\B^2)$, then, for every
$\alpha\in(0,1]$ and $z=(z_1,z_2)\in\B^2$, we have
\begin{equation}
\label{eq1}\big|z_1+g(z_2)-\frac{1}{\alpha}g(\alpha
z_2)\big|^2+|z_2|^2<\frac{1}{\alpha^2}.\end{equation}

Indeed, if $f(\B^2)$ is a starlike domain with respect to the origin, then, for
every $\alpha\in(0,1]$ and $z=(z_1,z_2)\in\B^2$, there exists
$z'=(z_1',z_2')\in\B^2$ such that $f(z')=\alpha f(z)$, i.e.
$z_2'=\alpha z_2$ and
$z_1'=\alpha z_1+\alpha g(z_2)-g(\alpha z_2)$. Rewriting the condition
$|z_1'|^2+|z_2'|^2<1$, we deduce $(\ref{eq1})$.

Now, suppose that $g$ is not bounded and fix $\alpha_0\in (0,1)$. We
deduce that
$\sup_{\zeta\in\D}\big|g(\zeta)-\frac{1}{\alpha_0}g(\alpha
_0 \zeta)\big|=\infty$. Hence $(\ref{eq1})$ does not hold for
$z=(0,\zeta)$, $\zeta\in\D$, and thus $f$ is not starlike.
\end{proof}

\begin{definition}[see \cite{And-Lemp, ABW2}]
\label{d2}
A domain $D\subset\Cn$ is called a starshapelike domain if there is
an automorphism $\psi:\Cn\to\Cn$  such that $\psi(D)$ is a starlike
domain with respect to the origin.
A mapping $f\in S(\Bn)$ is called starshapelike if $f(\Bn)$ is a
starshapelike domain.
\end{definition}

Using similar arguments as in the case of Propositions $\ref{p5}$ and
$\ref{p-new1}$, we deduce the following result.

\begin{proposition}
\label{p7}
If $\,\sum_{k=2}^{\infty} k|a_k|<\infty$, then $f$ is
starshapelike.
\end{proposition}

\section{The proof of Theorem~\ref{main}}

We start by proving the new growth estimate on the differential of functions in $S^0(\B^n)$, which we stated in the Introduction.

\begin{proof}[Proof of Proposition~\ref{dif-ineq}]
Let $\rho\in(0,1)$ be arbitrary. Since $\|f(\zeta)\|\le
\frac{\|\zeta\|}{(1-\|\zeta\|)^2}$ for all $\zeta\in\Bn$ (see e.g.
\cite{GHK-CJM}), we see that 
$F(\zeta):=\rho^{-1}(1-\rho)^2f(\rho\zeta)$, $\zeta\in\Bn$, is
a holomorphic self-map of~$\Bn$. In view of \cite[Theorem~4.6]{HK} (cf.
\cite[Lemma 3]{ChGa}), we deduce that:
$$\big\|\di{F}{\zeta}\big\|\le \frac{1}{1-\|\zeta\|^2},\quad
\zeta\in\Bn.$$

After elementary computations, we get:
\begin{equation}
\label{dif-ineq-e}
\big\|\di{f}{\rho\zeta}\big\|\le
\frac{1}{(1-\rho)^2(1-\|\zeta\|^2)},\quad \zeta\in\Bn.
\end{equation}

Let $r\in(0,1)$. Substituting $\sqrt{r}$ for $\rho$ in
(\ref{dif-ineq-e}), we obtain:
$$\big\|\di{f}{\sqrt{r}\zeta}\big\|\le
\frac{1}{(1-\sqrt{r})^2(1-r)},\quad \|\zeta\|\le \sqrt{r},$$
and thus,  replacing $\sqrt{r}\zeta$ with $z$, we have:
\begin{equation}\label{EQ_gr-est}
\big\|\di{f}{z}\big\|\le \frac{(1+\sqrt{r})^2}{(1-r)^3},\quad
\|z\|\le r,
\end{equation}
and we are done.
\end{proof}

With the aid of Propositions~\ref{dif-ineq} and~\ref{p4}, we are now ready to prove Theorem~\ref{main}. 
\begin{proof}[Proof of Theorem~\ref{main}] As mentioned in the Introduction, $A_0(\C^2)\circ S^0(\B^2)\subset S_R(\B^2)$. In order to see that the inclusion is strict, consider the shearing map $f:\B^2\to\C^2$ defined by
$$f(z):=(z_1+z_2^2 h(z_2),z_2),\;z=(z_1,z_2)\in\B^2,\quad \text{where~}~h(\zeta):=\exp(i/(1-\zeta)^3),\; \zeta\in\D.$$
According to Proposition~\ref{p4}, $f\in S_R(\B^2)$. Let us show that nevertheless  $f$ does not embed into a Loewner chain with range $\C^2$. By Theorem \ref{t1}, the latter is equivalent to $f\not\in A_0(\C^2)\circ S^0(\B^2)$.
 
Suppose on the contrary that there exists a normalized automorphism
$\Phi:\C^2\to\C^2$ such that $\Phi^{-1}\circ f\in
S^0(\B^2)$.

By elementary computations, we obtain:
$$\di{f}{z}=\frac{1}{(1-z_2)^{4}}
\begin{bmatrix}
 0 & 3iz_2^2h(z_2)  \\[1ex]
 0 & 0  
\end{bmatrix} 
 +
\begin{bmatrix}
1 & 0  \\[1ex]
0 & 1  
\end{bmatrix}
+
\begin{bmatrix}
0 & 2z_2 h(z_2)  \\[1ex]
0 & 0
\end{bmatrix},\quad z=(z_1,z_2)\in\B^2.$$
Observe that:
\begin{equation}
\label{bdd1}
|h(r)|=1,\quad r\in(0,1).
\end{equation}
Therefore, for all $r\in(1/2,1)$,
$$\|\di{f}{0,r}\|\ge \frac{3r^2}{(1-r)^4}-(1+2r)\ge\frac{2r^2}{(1-r)^4}.$$

By \eqref{bdd1}, we have that $r\mapsto f(0,r)$ is bounded on $(0,1)$. Taking into account that $\Phi$ is an automorphism of~$\C^2$, it follows that there exists a compact set $K\subset \C^2$ such that
$\Phi^{-1}(f(0,r))\in K$ for all $r\in(0,1)$. Let
$$C:=\max_{w\in K}\big\|\di{\Phi}{w}\big\|.$$

On the other hand, we have
$$\di{f}{z}=\di\Phi{\Phi^{-1}(f(z))}\,\di{(\Phi^{-1}\circ
f)}{z},\quad z\in\B^2,$$
and hence, we deduce that:
\begin{equation*}
\big\|\di{(\Phi^{-1}\circ f)}{0,r}\big\| \ge
\frac{1}{\big\|\di{\Phi}{\Phi^{-1}(f(0,r))}\big\|}\big\|\di{f}{0,r}\big\| \ge 
\frac{2r^2}{C(1-r)^4},\quad r\in(0,1/2).
\end{equation*}

On the other hand, $\Phi^{-1}\circ f\in S^0(\B^2)$ and hence, by~\eqref{EQ_gr-est},  
$$\big\|\di{(\Phi^{-1}\circ f)}{0,r}\big\|\le\frac{4}{(1-r)^3},\quad
r\in(0,1),$$
and thus we have arrived to an obvious contradiction, which completes the proof.
\end{proof}

\end{document}